\documentclass[12pt]{article}
\usepackage[utf8]{inputenc}
\usepackage[T1]{fontenc}
\usepackage{units}
\usepackage{url}
\usepackage[hmargin=1in]{geometry} 
\usepackage{amsmath,enumitem,array,amsfonts}
\newcolumntype{o}{@{}>{{}}c<{{}}@{}}
\usepackage{graphicx}
\usepackage{amsthm}
\usepackage{amssymb}
\usepackage{color}
\usepackage{ textcomp }
\usepackage{hyperref}
\usepackage{float}
\usepackage{pdfpages}
\usepackage{mathrsfs}
\usepackage{multicol}
\usepackage{verbatim}
\usepackage[all]{xy}
\usepackage{tikz-cd}
\usepackage{mathtools}
\usepackage{tensor}
\usepackage{ stmaryrd }
\usepackage{mathtools}
\usepackage{thm-restate}
\usepackage{thmtools}
\usepackage{hyperref}
\usepackage{cleveref}
\usepackage{fancyhdr}
\usepackage{atbegshi}
\usepackage{stackengine}
\usepackage{mathtools}
\usepackage{subcaption}
\usepackage{cleveref}

\DeclareMathOperator{\Hom}{Hom}

\newcommand{\im}{\mathrm{im}}

\newcommand{\Ann}{\mathrm{Ann}}

\newcommand{\hsl}[2]{\mathrm{HSL}_{#1}(#2)}

\newcommand{\gr}[1]{\mathrm{gr}\left(#1\right)}
\newcommand{\fte}[1]{\mathrm{Fte}\left(#1\right)}

\newcommand{\fm}{\mathfrak{m}}
\newcommand{\fn}{\mathfrak{n}}

\newcommand{\fq}{\mathfrak{q}}

\newcommand{\bbN}{\mathbb{N}}

\newcommand{\bbR}{\mathbb{R}}

\newcommand{\bbZ}{\mathbb{Z}}

\newcommand{\cF}{\mathcal{F}}
\newcommand{\cG}{\mathcal{G}}
\newcommand{\cH}{\mathcal{H}}

\newcommand{\cP}{\mathcal{P}}

\theoremstyle{definition}
\newtheorem*{theorem*}{Theorem}
\newtheorem*{solution*}{Solution}
\newtheorem*{problem*}{Problem}

\newtheorem*{definition*}{Definition}
\newtheorem*{lemma*}{Lemma}
\newtheorem*{corollary*}{Corollary}
\newtheorem*{proposition*}{Proposition}
\newtheorem{theorem}{Theorem}[section]
\newtheorem{definition}[theorem]{Definition}
\newtheorem{remark}[theorem]{Remark}
\newtheorem{example}[theorem]{Example}
\newtheorem{proposition}[theorem]{Proposition}
\newtheorem{corollary}[theorem]{Corollary}

\newtheorem{lemma}[theorem]{Lemma}

\newtheorem{notation}[theorem]{Notation}

\newcommand{\ceil}[1]{\lceil #1\rceil}

\newcommand{\nq}[1]{N_Q^{(#1)}}

\author{Havi Ellers}
\title{A bound on the Hartshorne-Speiser-Lyubeznik number of semigroup rings}
\date{}

\newcounter{problem}

\begin{document}

\maketitle	
\begin{abstract}
	In this paper we prove an explicit, computable upper bound on the Hartshorne-Speiser-Lyubeznik number of the local cohomology of a pointed, affine semigroup ring over a perfect field of positive characteristic. This bound depends only on the characteristic of the ring and properties of the semigroup. 
\end{abstract}

\section{Introduction}\label{intro}
\thispagestyle{fancy}
\fancyhead{}
\fancyfoot[LO,RE]{\noindent\rule{2cm}{0.4pt}\\ The author was supported by NSF grants DMS 2200501, DMS 2101075 and DMS 1840234.}

The Hartshorne-Speiser-Lyubeznik number (HSL number) is a numerical invariant of modules with a Frobenius action. Roughly speaking, it is a degree of nilpotency for the Frobenius action on the module. One important example of a class of modules with a Frobenius action is the class of local cohomology modules of a ring of positive characteristic. If the ring is local, we can define the Hartshorne-Speiser-Lyubeznik number (HSL number) of the ring to be the largest HSL number of any of its local cohomology modules with support at the maximal ideal. Then the Hartshorne-Speiser-Lyubeznik number can be seen as a singularity invariant of the ring, and can be connected to various types of $F$-singularities. See \cite{hq19,mad19,Quy19,PQ19,hq21,kmps,maddoxPandey,st17} for several of these connections. In particular, a local ring of positive characteristic is $F$-injective if and only if it has zero HSL number. 

Let $R$ be a commutative Noetherian ring of characteristic $p>0$. A \textit{Frobenius action} on an $R$-module $M$ is an additive map $\rho:M\to M$ such that $\rho(rm)=r^p\rho(m)$ for all $r\in R,m\in M$. Given any $R$-module $M$ with Frobenius action $\rho:M\to M$, we can define the \textit{nilpotent submodule} of $M$ to be
\begin{align*}
	0^\rho_M &= \{m\in M\,|\,\rho^e(m)=0\text{ for some }e\in\bbN\}.
\end{align*}

The Hartshorne-Speiser-Lyubeznik number of $M$ is then defined as follows.

\begin{definition}
	Let $M$ be an $R$-module with Frobenius action $\rho:M\to M$. Then the \textit{Hartshorne-Speiser-Lyubeznik number} of $M$ is
	\begin{align*}
		\hsl{}{M} &= \inf\{e\in\bbN\,|\,\rho^e(m)=0\text{ for all }m\in0^\rho_M\}.
	\end{align*} 
\end{definition}

If $M$ is Noetherian it is straightforward to see that $\hsl{}{M}$ is finite. Surprisingly, the HSL number is also finite for Artinian modules:
\begin{theorem}\label{thing9}
	(\cite[Prop 1.11]{hs77}, \cite[Prop 4.4]{lyu97}, \cite[Cor 1.8]{sha06}) If $M$ is an Artinian $R$-module with Frobenius action $\rho:M\to M$, then $\hsl{}{M}<\infty$. That is, there exists an $e\in\bbN$ such that $\rho^e(m)=0$ for all $m\in0^\rho_M$.
\end{theorem}
One important class of Artinian $R$-modules with a Frobenius action is the class of local cohomology modules of a local ring $R$ with support at the maximal ideal, and with Frobenius action induced by the natural Frobenius action on $R$. In this paper, we focus on affine pointed semigroup rings $R$ over a field $k$ of characteristic $p>0$ with maximal monomial ideal $\fm$. We find an explicit upper bound on $\hsl{}{H^\ell_\fm(R)}$ for any $\ell\in\bbN$.

\begin{theorem}\label{thing0}
	Let $Q$ be an affine pointed semigroup of dimension $n$. There exists $N_Q\in\bbN$, depending only on $Q$, such that the following holds. Let $k$ be any perfect field of any characteristic $p>N_Q$ and let $\ell\in\bbN$. Let $\fm$ be the maximal monomial ideal of $k[Q]$. For $m_\cH$ as described in Remark \ref{thing16}, 
	\begin{align*}
		\hsl{}{H^\ell_\fm(k[Q])}\le\max_{\cH}\lceil\log_p m_{\cH}\rceil,
	\end{align*}
	where the maximum ranges over faces $\cH$ of $Q$ with $\dim\cH\ge\ell$.
\end{theorem}

\begin{remark}
	The number $m_\cH$ depends only on $\cH$, and roughly speaking it depends on how far $\cH$ is from its saturation.
\end{remark}

\begin{remark}
	See Definition \ref{thing56} and Remark \ref{thing17} for an explicit description of $N_Q$. 
\end{remark}

\begin{remark}
	Theorem \ref{thing0} can be refined so that $N_Q$ is optimized for each $\ell$. See Theorem \ref{thing00}. 
\end{remark}

\noindent Theorem \ref{thing0} has the following consequence:
\begin{corollary}
	Let $Q$ be an affine pointed semigroup of dimension $n$. Then for any $p\gg0$, if $k$ is any perfect field of characteristic $p$ and $\ell\in\bbN$ then $\hsl{}{H^\ell_\fm(k[Q])}$ is either 0 or 1. That is, either Frobenius is injective on $H^\ell_\fm(k[Q])$ or the nilpotent submodule of $H^\ell_\fm(k[Q])$ is exactly the kernel of Frobenius.
\end{corollary}
Note that if the semigroup $Q$ is saturated then $k[Q]$ is normal (see for example \cite[Prop 7.25]{combca}) and strongly $F$-regular (see \cite[Example 4.20]{fbook}) and therefore Frobenius acts injectively on local cohomology and $\hsl{}{H^\ell_\fm(k[Q])}=0$ for all indices $\ell$. However, the converse is false: there are non-saturated semigroups $Q$ such that $\hsl{}{H^\ell_\fm(k[Q])}=0$ for all indices $\ell$ and all $p\gg0$. For example, using Theorem \ref{thing0} the reader can easily verify the following example.

\begin{example}
	Let $R=k[x^2,x^2y,xy,xy^2,y^2]$ and $Q$ be the associated semigroup in $\bbZ^2$. Then $x=\frac{x^2y}{xy}$ is in the saturation of $Q$ but not in $Q$, hence $Q$ is not saturated. However, if $p>2$ then $\hsl{}{H^1_\fm(k[Q])}=\hsl{}{H^2_\fm(k[Q])}=0$.
\end{example}

In Section \ref{fte} we additionally see how the results of this paper can be used to obtain bounds on the Frobenius test exponent of any weakly $F$-nilpotent or Cohen-Macaulay affine pointed semigroup ring. See Corollary \ref{thing8} and Corollary \ref{thing11}.

The organization of the paper is as follows. In Section \ref{background} we give background on semigroup rings, local cohomology, and HSL numbers. In Section \ref{prelimresul} we give some preliminary results and in Section \ref{mainthm} we give the proof of the main theorem. In Section \ref{fte} we describe applications to bounds for Frobenius test exponents.

\section{Background}\label{background}

\subsection{Semigroup Rings}

A \textit{semigroup} is a set $Q$ together with a binary operation $+$ that satisfies the associative property: for all $u,v,w\in Q$ we have $(u+v)+w=u+(v+w)$. A semigroup is \textit{affine} if it is finitely generated and is isomorphic to a sub-semigroup of the free abelian group $\bbZ^d$ \cite[page 50]{bg09}. We always assume that our semigroups are affine. 

Given a semigroup $Q$, for each field $k$ we can form the corresponding \textit{semigroup ring} $k[Q]$, which is the $k$-algebra with $k$-basis $\{\underline{x}^v\,|\,v\in Q\}$ and multiplication defined by $\underline{x}^v\cdot\underline{x}^w=\underline{x}^{v+w}$.  Note that $k[Q]$ has a natural $Q$-grading. If $Q$ is affine then $k[Q]$ is finitely generated as a $k$-algebra.

The \textit{maximal monomial ideal} of $k[Q]$ is the ideal generated by all non-unit monomials. A semigroup is \textit{pointed} if it has no non-identity units. The maximal monomial ideal of $k[Q]$ is a maximal ideal if and only if $Q$ is pointed (see \cite[Exercise 20.15]{24hrs}). When $Q$ is pointed we will also refer to $k[Q]$ as a \textit{pointed} semigroup ring.

Let $M$ be a lattice, i.e.\ a finitely generated torsion-free abelian group. Let $M^\vee=\Hom_\bbZ(M,\bbZ)$ denote the dual lattice of $M$, which is itself a lattice. Let $M_\bbR$ denote the $\bbR$-vector space $M\otimes_\bbZ\bbR$. We denote by $M_\bbR^*$ the usual vector space dual. A \textit{rational convex polyhedral cone} in $M_\bbR$ is a set
\begin{align*}
	\sigma &= \{r_1v_1+\cdots+r_sv_s\in M_\bbR\,|\,r_i\ge0\}
\end{align*} 
generated by a finite set of vectors $v_1,\dots,v_s\in M$. The \textit{dual} of $\sigma$ is
\begin{align*}
	\sigma^\vee &= \{u\in M_\bbR^*\,|\,u(v)\ge0\text{ for all }v\in\sigma\}.
\end{align*}
Note that the dual of a rational convex polyhedral cone in $M_\bbR$ is a rational convex polyhedral cone in $M_\bbR^*$. The \textit{dimension} of a cone is the dimension of the $\bbR$-vector space $\bbR\cdot\sigma=\sigma+(-\sigma)$.

Given a lattice $M$ and a rational convex polyhedral cone $\sigma$ in $M_\bbR$, a \textit{face} $\tau$ of $\sigma$ is $\{v\in\sigma\,|\,u(v)=0\}$ for some $u\in\sigma^\vee$. A face is itself a rational convex polyhedral cone, and the \textit{dimension} of a face is its dimension as a cone. A \textit{facet} is a face of codimension 1.

If $\sigma$ spans $M_\bbR$ and $\tau$ is a facet of $\sigma$ then there is a $u\in\sigma^\vee$, unique up to scalar multiplication, such that $\tau=\{v\in M_\bbR\,|\,u(v)=0\}$. We denote such a vector by $u_\tau$. Then we have the following proposition.
\begin{proposition}(See \cite[page 11]{fulton}.)
	Let $\sigma$ be a rational convex polyhedral cone in $M_\bbR$. If $\sigma$ spans $M_\bbR$ and $\sigma\not=M_\bbR$, then $\sigma$ is the intersection of the half-spaces $H_\tau=\{v\in M_\bbR\,|\,u_\tau(v)\ge0\}$ as $\tau$ ranges over the facets of $\sigma$.
\end{proposition}

Given a semigroup $Q$, let $\gr{Q}$ denote the group generated by $Q$ (see \cite[page 50]{bg09} for a construction of $\gr{Q}$; if $Q$ is a sub-semigroup of $\bbZ^d$ then $\gr{Q}$ can be identified with $\bbZ Q\subset\bbZ^d$). Then $\gr{Q}$ is a lattice. If $M=\gr{Q}$ and $\sigma$ is the cone generated in $M_\bbR$ by $Q$, we can, similarly to above, define a \textit{face} of $Q$ to be the intersection of $Q$ with a face of $\sigma$, and a \textit{facet} of $Q$ to be the intersection of $Q$ with a facet of $\sigma$. The \textit{dimension} of a face of $Q$ is the dimension of the corresponding face of $\sigma$. We will also refer to a face of $Q$ of dimension $i$ as an $i$-face. A \textit{ray} is a 1-face. We further define the \textit{saturation} of $Q$ to be
\begin{align*}
	Q_{\mathrm{sat}}=\sigma\cap M.
\end{align*} 
Note that $Q_{\mathrm{sat}}$ is also a semigroup, and that $k[Q_{\mathrm{sat}}]$ is the normalization of $k[Q]$ (see \cite[Prop 7.25]{combca}).

\subsection{Local Cohomology and HSL Numbers}\label{thing12}

For a commutative ring $R$ with ideal $I\subset R$, the \textit{$i^{\mathrm{th}}$ local cohomolgy module} of $R$ at $I$ is the $i^{\mathrm{th}}$ right derived functor of 
\begin{align*}
	H^0_I(R)=\{x\in R\,|\,I^Nx=0\text{ for some }N\in\bbN\}
\end{align*}
in the category of $R$-modules. We consider $R=k[Q]$ a pointed semigroup ring and $I=\fm$ the maximal monomial ideal of $R$. In this case, the local cohomology of $R$ at $\fm$ can be computed using a combinatorial complex called the Ishida complex (see \cite{ish88}). 

We use the following notation from \cite{24hrs}: for a face $\cF$ of $Q$, we write $k[Q]_\cF$ for the localization of $k[Q]$ at the set of monomials $\underline{x}^v$ for $v\in \cF$. Note that $k[Q]_\cF$ is equivalently the semigroup ring generated by the semigroup $$Q-\cF=\{u-v\,|\,u\in Q,v\in\cF\}.$$ The \textit{Ishida complex} $\mho^\bullet_Q$ of the semigroup $Q$ is the complex
\begin{align}
	0 \to k[Q]\to\bigoplus_{\text{rays }\cF}k[Q]_\cF\to\cdots\to\bigoplus_{i\text{-faces }\cF}k[Q]_\cF\xrightarrow{\delta^i}\cdots\to\bigoplus_{\text{facets }\cF}k[Q]_\cF\to k[M]\to0, \label{thing6}
\end{align}
where $M$ is the group generated by $Q$.

The differential $\delta$ is described via componentwise maps $\delta_{\cF,\cG}:k[Q]_\cF\to k[Q]_\cG$, for $\cF$ an $i$-face and $\cG$ an $(i+1)$-face. If $\cF\not\subset\cG$ then $\delta_{\cF,\cG}$ is the zero map. Otherwise $\delta_{\cF,\cG}$ is, up to a sign, the natural localization $k[Q]_\cF\to k[Q]_\cG$. The signs will not be relevant in this paper, but are derived as in the algebraic cochain complex for the polytope $\cP$ obtained as a transverse hyperplane section of $\sigma$. See \cite[page 208]{24hrs}, \cite[Def 13.21]{combca}, \cite[Section 2]{moy22} and \cite{ish88} for more extensive descriptions. Note that the $H^i_\fm(k[Q])$ are all $M$-graded, since the $k[Q]_\cF$ are $M$-graded and the differential $\delta$ preserves that grading.

There is a natural Frobenius action on the complex \eqref{thing6}, given by 
\begin{align*}
	\bigoplus_{i\text{-faces }\cF}k[Q]_\cF&\to\bigoplus_{i\text{-faces }\cF}k[Q]_\cF\\
	\bigoplus_{i\text{-faces }\cF}\frac{f_\cF}{x^{w_\cF}}&\mapsto \bigoplus_{i\text{-faces }\cF}\frac{f_\cF^p}{x^{pw_\cF}}
\end{align*}
and this induces a corresponding Frobenius action on local cohomology. We will use this description of the Frobenius action on $H^i_\fm(k[Q])$ to study their HSL numbers.

\section{Preliminary Results}\label{prelimresul}

We are interested in semigroup rings $k[Q]$ where $k$ is a perfect field of positive characteristic and $Q$ is a pointed semigroup containing identity. We will use the following notation.

\begin{notation}\label{thing101} Let $Q$ be an affine pointed semigroup, and let $M=\gr{Q}\cong\bbZ^n$ be the group generated by $Q$. Let $M_\bbR$ be the $\bbR$-vector space $M\otimes_\bbZ\bbR$. Let $\sigma$ be the cone generated by $Q$ in $M_\bbR$; note that $\dim(\sigma)=n$, i.e.\ $\sigma$ spans $M_\bbR$. Let $Q_{\mathrm{sat}}:=\sigma\cap M$ be the saturation of $Q$. Let $\sigma_1,\dots,\sigma_r$ denote the facets of $\sigma$, let $u_1,\dots,u_r\in\sigma^\vee$ be such that $\sigma_i=\{v\in\sigma\,|\,u_i(v)=0\}$, and write
	\begin{align*}
		\sigma &= \bigcap_{i=1}^r\{v\in M_\bbR\,|\,u_i(v)\ge0\}. 
	\end{align*}
	For a face $\cH$ of $Q$, let $\sigma_{\cH}=\bbR_{\ge0}\cH$ denote the face of $\sigma$ generated by $\cH$. Finally, let $\fm$ be the maximal monomial ideal of $k[Q]$.
\end{notation}

\begin{lemma}\label{thing13} 
	(See \cite[Exercise 7.15]{combca}.) Let $Q$ be an affine semigroup. Then $Q$ contains a translate of its saturation: $\gamma_Q+Q_{\mathrm{sat}}\subset Q$ for some $\gamma_Q\in Q$.
\end{lemma}

\begin{proof} By \cite[Prop 7.25 and Cor 13.13]{combca}, $k[Q]\hookrightarrow k[Q_{\mathrm{sat}}]$ is a module-finite extension. Say $k[Q_{\mathrm{sat}}]$ is generated over $k[Q]$ by $g_1,\dots,g_m$. Without loss of generality we may assume that the $g_i$ are monomials, i.e.\ $g_i=\underline{x}^{a_i}$ for some $a_i\in Q_{\mathrm{sat}}$. Since $\mathrm{frac}(k[Q])=\mathrm{frac}(k[Q_{\mathrm{sat}}])$ we can also write
	\begin{align}\label{thing5}
		g_i=\frac{f_i}{h_i}
	\end{align}
	for some $f_i,h_i\in k[Q]$. We claim that we may assume $f_i,h_i$ are also monomials. Indeed, write
	\begin{align*}
		h_i &= \sum_{j=1}^t b_{ij}\underline{x}^{c_{ij}}\\
		f_i &= \sum_{\ell=1}^s d_{i\ell}\underline{x}^{e_{i\ell}}
	\end{align*}
	for some $b_{ij},d_{i\ell}\in k,c_{ij},d_{i\ell}\in Q$. Rearranging \eqref{thing5} we get
	\begin{align*}
		f_i &= h_ig_i, 
	\end{align*}
	that is,
	\begin{align*}
		\sum_{\ell=1}^s d_{i\ell}\underline{x}^{e_{i\ell}} &= \sum_{j=1}^t b_{ij}\underline{x}^{a_i+c_{ij}}.
	\end{align*}
	Hence without loss of generality we may assume $s=t$ and $d_{i\ell}\underline{x}^{e_{i\ell}}=b_{i\ell}\underline{x}^{a_i+c_{i\ell}}$ for all $\ell$. Thus in particular
	\begin{align*}
		g_i &= \underline{x}^{a} = \frac{d_{i1}\underline{x}^{e_{i1}}}{b_{i1}\underline{x}^{c_{i1}}},
	\end{align*}	
	and so we may assume that $f_i,h_i$ are monomials in $k[Q]$ as desired. Now write
	\begin{align*}
		h_i&=b_i\underline{x}^{c_i}
	\end{align*}
	for some $c_i\in Q$. Then $\underline{x}^{c_1+\cdots+c_m}\in k[Q]$ and $\underline{x}^{c_1+\cdots+c_m}\cdot k[Q_{\mathrm{sat}}]\subset k[Q]$. Hence $c_1+\cdots+c_m\in Q$ and $(c_1+\cdots+c_m)+Q_{\mathrm{sat}}\subset Q$ as desired. Setting $\gamma_Q=c_1+\cdots+c_m$ we are done.
\end{proof}

\begin{definition}\label{thing102}
	With notation \ref{thing101}, let $\gamma_Q\in Q$ be such that $\gamma_Q+Q_{\mathrm{sat}}\subset Q$ as in Lemma \ref{thing13}. We then define $m_Q:=\max_{1\le i\le r}\{u_i(\gamma_Q)\}$.
\end{definition}

\begin{lemma}\label{thing1}
	With Notation \ref{thing101} and Definition \ref{thing102}, if $v\in\sigma^{\mathrm{int}}\cap M$ then $mv\in Q$ for $m\ge m_Q$.
\end{lemma}

\begin{proof}
	We can write
	\begin{align*}
		\sigma^{\mathrm{int}}\cap M&= \bigcap_{i=1}^r\{v\in M\,|\,u_i(v)>0\},\\
		\gamma_Q+Q_{\mathrm{sat}}&=\bigcap_{i=1}^r\{v\in M\,|\,u_i(v)\ge u_i(\gamma_Q)\}.
	\end{align*}
	Say $v\in\sigma^{\mathrm{int}}\cap M$. Then for $m\ge m_Q$ and $1\le i\le r$ we have
	\begin{align*}
		u_i(mv)=mu_i(v)\ge u_i(\gamma_Q)u_i(v)\ge u_i(\gamma_Q),
	\end{align*}
	since $u_i(v)\ge1$. Hence $mv\in\gamma_Q+Q_\mathrm{sat}\subset Q$.
\end{proof}

\begin{lemma}\label{thing55}
	Let $\cH$ be a face of $Q$ and let $N\in\bbZ$ be the unique generator  of the annihilator of the finite abelian group
	\begin{align*}
		\frac{\gr{\sigma_\cH\cap M}}{\gr{\sigma_\cH\cap Q}},
	\end{align*}
	or zero if the group is trivial. Fix a prime $p>N$. Then for $v\in M$, if $p^ev\in\gr{\sigma_\cH\cap Q}$ for some $e\in\bbN$ and $v\in\gr{\sigma_\cH\cap M}$, then $v\in\gr{\sigma_\cH\cap Q}=\gr{\cH}$.
\end{lemma}

\begin{example}\label{thing14}
	 Let $Q$ be the semigroup generated in $\bbZ^2$ by 
	\begin{align*}
		v_1 = \begin{bmatrix}
			3\\
			0
		\end{bmatrix}, \qquad v_2=\begin{bmatrix}
			1\\
			1
		\end{bmatrix}, \qquad v_3=\begin{bmatrix}
			0\\
			1
		\end{bmatrix}.
	\end{align*}\\
	\begin{figure}[H]
		\centering
		\begin{tikzpicture}
			\draw (0,0) -- coordinate (x axis mid) (6,0); 
			\draw (0,0) -- coordinate (y axis mid) (0,6); 
			\foreach \x in {0,...,6}
			\draw (\x,1pt) -- (\x,-3pt)
			node[anchor=north] {\x};
			\foreach \y in {0,...,6}
			\draw (1pt,\y) -- (-3pt,\y) 
			node[anchor=east] {\y}; 
			\filldraw [black] (0,0) circle (2pt);
			\foreach \x in {0,3,6}
			\filldraw [black] (\x,0) circle (2pt);
			\foreach \x in {0,1,3,4,6}
			\filldraw [black] (\x,1) circle (2pt);
			\foreach \x in {0,...,6}
			\foreach \y in {2,...,6}
			\filldraw [black] (\x,\y) circle (2pt);
		\end{tikzpicture}
		\caption*{$Q$}
	\end{figure}

	Then $\gr{Q}=\bbZ^2$, $M_\bbR=\bbR^2,$ and $\sigma=\mathrm{span}_{\bbR_{\ge0}}\{v_1,v_2,v_3\}$ is the first quadrant of $\bbR^2$. Let $\sigma_1$ be the $x$-axis, which is a facet of $\sigma$. Let $\cH=\sigma_1\cap Q$. Note that $\gr{\sigma_1\cap M}$ is the set of integer-valued points on the $x$-axis, but $\gr{\sigma_1\cap Q}$ is the set of integer-valued multiples of three on the $x$-axis. Hence
	\begin{align*}
		\frac{\gr{\sigma_1\cap M}}{\gr{\sigma_1\cap Q}}\cong\bbZ/3\bbZ
	\end{align*}
	and so $N=3$ for this choice of $Q$ and $\cH$.
\end{example}

\begin{proof}[Proof of Lemma \ref{thing55}]
	Let $\overline{v}$ denote the image of $v$ in $\frac{\gr{\sigma_\cH\cap M}}{\gr{\sigma_\cH\cap Q}}$. Since $p^ev\in\gr{\sigma_\cH\cap Q}$, $p^e\in\Ann_\bbZ(\overline{v})$. But also 
	\begin{align*}
		N\in N\bbZ=\Ann_\bbZ\left(\frac{\gr{\sigma_\cH\cap M}}{\gr{\sigma_\cH\cap Q}}\right)\subset\Ann_\bbZ(\overline{v}).
	\end{align*}
	Hence $p^e,N\in\Ann_\bbZ(\overline{v})$. But since $p>N$ is prime, $\gcd(p^e,N)=1$. Thus $\Ann_\bbZ\overline{v}=\bbZ$, i.e., $\overline{v}=0$. Hence $v\in\gr{\sigma_\cH\cap Q}$ as desired.
\end{proof}

\section{Main Result}\label{mainthm}

In this section we prove the main result, Theorem \ref{thing0}. Throughout we use notation as in Notation \ref{thing101}. First we must define a certain invariant.

\begin{definition}\label{thing56}
	For $Q$ an affine pointed semigroup and $\ell\in\bbN$, we define $\nq{\ell}\in\bbZ$ to be the maximum annihilator of any of the groups 
	\begin{align*}
		\frac{\gr{\sigma_\cH\cap M}}{\gr{\sigma_\cH\cap Q}}
	\end{align*}
	where $\cH$ is a face of $Q$ with $\dim\cH\ge\ell-1$, or zero if all these groups are trivial.
\end{definition}

\begin{remark}\label{thing16}
	Recall from Section \ref{prelimresul} that for a face $\cH$ of $Q$ the number $m_\cH$ is calculated as follows. Write $\sigma_{\cH}$ as an intersection of half-spaces:
	\begin{align*}
		\sigma_{\cH} = \bigcap_{i=1}^{r_\cH}\{v\in\gr{\cH}\otimes_{\bbZ}\bbR\,|\,u_i^{\cH}(v)\ge0\}
	\end{align*}
	and let $\gamma_{\cH}\in\cH$ be such that $\gamma_{\cH}+\cH_{\mathrm{sat}}\subset\cH$. Then $m_\cH=\max_{1\le i\le r_\cH}\{u_i^\cH(\gamma_\cH)\}$.
\end{remark}

\noindent We will prove the following refinement of Theorem \ref{thing0}. 

\begin{theorem}\label{thing00}
	Let $Q$ be an affine pointed semigroup of dimesion $n$. Fix $\ell\in\bbN$, and let $\nq{\ell}$ be as in Definition \ref{thing56}. Let $k$ be any perfect field of any characteristic $p>\nq{\ell}$. Let $\fm$ be the maximal monomial ideal of $k[Q]$ and let $m_\cH$ be as in definition \ref{thing102}. Then the Hartshorne-Speiser-Lyubeznik number of $H^\ell_\fm(k[Q])$ is less than or equal to $\max_{\cH}\lceil\log_p m_{\cH}\rceil$, where the maximum ranges over faces $\cH$ of $Q$ with $\dim\cH\ge\ell$.
\end{theorem}

\begin{remark}\label{thing17}
	We have $\nq{1}\ge\cdots\ge\nq{n}$. Hence Theorem \ref{thing0} follows from Theorem \ref{thing00} by taking $N_Q$ to be $\nq{1}$.
\end{remark}

Before we prove Theorem \ref{thing00}, we present an example and several necessary intermediate results. The proof of Theorem \ref{thing00} starts on \cpageref{thing18}.

\begin{example}\label{thing15}
	Consider the semigroup $Q$ from Example \ref{thing14}. Here we calculate $\nq{\ell}$ and the upper bound from Theorem \ref{thing00} for $\hsl{}{H^\ell_\fm(k[Q])}$ for $\ell=1,2$.
	
	First we calculate $\frac{\gr{\sigma_\cH\cap M}}{\gr{\sigma_\cH\cap Q}}$ for each face $\cH$ of $Q$. The cone $\sigma$ has four faces: one 0-face $\sigma_0$, which is the origin, two 1-faces $\sigma_1$ and $\sigma_2$, which are the $x$- and $y$-axes respectively, and one 2-face, which is $\sigma$. Let $\cH_i=\sigma_i\cap Q$ for $i=0,1,2$. We have
	\begin{align*}
		\frac{\gr{\sigma_0\cap M}}{\gr{\sigma_0\cap Q}}=0/0=0, \qquad& \frac{\gr{\sigma_1\cap M}}{\gr{\sigma_1\cap M}} = \bbZ/3\bbZ\\
		\frac{\gr{\sigma_2\cap M}}{\gr{\sigma_2\cap Q}} = \bbZ/\bbZ=0,\qquad& \frac{\gr{\sigma\cap M}}{\gr{\sigma\cap Q}} = \bbZ^2/\bbZ^2=0
	\end{align*}
	It follows that $\nq{1}=\nq{2}=3$.
	
	Now we calculate $m_{\cH_i}$ for $i=1,2$. For ease of notation, let $m_i=m_{\cH_i}$ and $\gamma_i=\gamma_{\cH_i}$ for $i=1,2$. Since all the $\cH_i$ are saturated, we may take $\gamma_i=[0]$ for all $i$. Hence $m_i=0$ for all $i$ as well.
	
	Now we calculate $m_Q$. Let 
	\begin{align*}
		u_1 &= \begin{bmatrix}
			1\\
			0
		\end{bmatrix},\qquad u_2=\begin{bmatrix}
			0\\
			1
		\end{bmatrix}
	\end{align*}
	so that $\sigma_i=\{u\in\sigma\,|\,u_i(v)=0\}$ for $i=1,2$. Further, let $\gamma_Q=\begin{bmatrix}
		2\\
		0
	\end{bmatrix}$ so that $\gamma_Q+Q_{\mathrm{sat}}\subset Q$. Then
	\begin{align*}
		u_1(\gamma_Q)&=\begin{bmatrix}
			1\\
			0
		\end{bmatrix}\cdot\begin{bmatrix}
			2\\
			0
		\end{bmatrix}=2\\
		u_2(\gamma_Q) &= \begin{bmatrix}
			0\\
			1
		\end{bmatrix}\cdot \begin{bmatrix}
			2\\
			0
		\end{bmatrix} = 0
	\end{align*}
	and so $m_Q=\max\{0,2\}=2$.
	
	Now we can calculate the upper bound given in Theorem 4.1 for $\hsl{}{H^\ell_\fm(k[Q])}$ for $\ell=1,2$. From our above calculations, for $p>3$ we have
	\begin{align*}
		\hsl{}{H^2_\fm(k[Q])}&\le\max_{\dim\cH\ge2}\lceil\log_p m_\cH\rceil = \lceil\log_p2\rceil=1,\\
		\hsl{}{H^1_\fm(k[Q])}&\le\max_{\dim\cH\ge1}\lceil\log_p m_\cH\rceil = \max\{-\infty,\ceil{\log_p2}\} = \lceil\log_p2\rceil=1.
	\end{align*}
	Hence for $p>3$ and $\ell=1,2$ the submodule $0^F_{H^\ell_\fm(k[Q])}$ is just the kernel of the Frobenius action. 
\end{example}

We now turn to the proof of Theorem \ref{thing00}. It turns out that the proof will come down to the following fact about semigroups:

\begin{theorem}\label{thing01}
	Let $Q$ be an affine pointed semigroup of dimension $n$. Fix $\ell\in\bbN$ and let $\nq{\ell}$ be as in Definition \ref{thing56}. Fix a prime $p>\nq{\ell}$, let $\cF$ be an $(\ell-1)$-face of $Q$ and let $v\in \gr{Q},e\in\bbN$ be such that $p^ev\in Q-\cF$. Then $p^{\widetilde{e}}v\in Q-\cF$ for $\widetilde{e}\ge\max_{\cH}\lceil\log_p m_{\cH}\rceil$, where the maximum ranges over faces $\cH$ of $Q$ with $\dim\cH\ge\ell$ and where $m_\cH$ is as in Definition \ref{thing102}.
\end{theorem}

Before proving Theorem \ref{thing01} we first prove two additional lemmas. The proof of Theorem \ref{thing01} starts on \cpageref{thing19}.

\begin{lemma}\label{thing45}
	Fix $\ell\in\bbN$. Let $\cF$ be an $(\ell-1)$-face of $Q$ and let $v\in M,e\in\bbN$ be such that $p^ev\in Q-\cF$. Write $p^ev+w\in Q$ for some $w\in\cF$. Then $v+w\in\sigma$. 
\end{lemma}

\begin{proof}
	Fix $1\le j\le r$. We wish to show that $u_j(v+w)\ge0$. First suppose $u_j(v)\ge0$. Then
	\begin{align*}
		u_j(v+w)=u_j(v)+u_j(w)\ge 0
	\end{align*}
	since $w\in\cF\subset Q\subset\sigma$. Now suppose $u_j(v)<0$. Then
	\begin{align*}
		u_j(v+w)=u_j(v)+u_j(w)\ge p^eu_j(v)+u_j(w) = u_j(p^ev+w)\ge0
	\end{align*}
	since $p^ev+w\in Q\subset Q$. In either case we get $v+w\in\sigma$, as desired.
\end{proof}

\begin{lemma}\label{thing46}
	Fix $\ell\in\bbN$. Let $\cF$ be an $(\ell-1)$-face of $Q$ and let $v\in M,e\in\bbN$ be such that $p^ev\in Q-\cF$. Then there is a face $\cH$ of $Q$ and $w\in\cF$ such that $\cF\subset\cH,p^ev+w\in Q$ and $v+w\in\sigma_{\cH}^{\mathrm{int}}$.
\end{lemma}

\begin{proof}
	To prove Lemma \ref{thing46} it suffices to show the following statement:
	\begin{quotation}
		\noindent Let $\cF$ be an $(\ell-1)$-face of $Q$ and let $v\in M,e\in\bbN$ be such that $p^ev\in Q-\cF$. Let $w\in\cF$ be such that $p^ev+w\in Q$. Then there is a face $\widetilde{\cH}$ of $Q$ and $\widetilde{w}\in\cF$ such that $\cF\subset\widetilde{\cH}$ and $v+w+\widetilde{w}\in\sigma_{\widetilde{\cH}}^{\mathrm{int}}$.
	\end{quotation}
	To prove this, first recall that $v+w\in\sigma$ by Lemma \ref{thing45}. We proceed by induction on the number, $d$, of facets of $\sigma$ containing $v+w$.
	
	For the base case, suppose $d=0$. Then $v+w$ is not contained in any facets of $\sigma$, i.e., $v+w\in\sigma^{\mathrm{int}}=\sigma_Q^{\mathrm{int}}$. Since $\cF\subset Q$ we can set $\widetilde{\cH}=Q$ and $\widetilde{w}=0\in\cF$ and we're done.
	
	Now suppose the claim holds for natural numbers strictly less than $d$. Suppose $v+w$ is contained in $d$ facets of $\sigma$. Let $\cH$ be the minimal face of $Q$ such that $v+w\in\sigma_\cH$, i.e.,
	\begin{align*}
		\sigma_\cH=\bigcap_{v+w\in\sigma_j}\sigma_j,\qquad \cH=\sigma_\cH\cap Q.
	\end{align*}
	If $\cF\subset\cH$ we're done, since $v+w\in\sigma_{\cH}^{\mathrm{int}}$ by minimality of $\cH$. So suppose $\cF\not\subset\cH$. We claim there is $w_1\in\cF$ such that $v+w+w_1$ is contained in strictly fewer facets than $v+w$. We can then apply the induction hypothesis.
	
	Indeed, since $\cF\not\subset\cH$, there is $w_1\in\cF$ such that $w_1\not\in\cH$. Further,
	\begin{align*}
		\cH=\sigma_\cH\cap Q=\left(\bigcap_{v+w\in\sigma_j}\sigma_j\right)\cap Q
	\end{align*}
	so there is a facet $\sigma_{j_1}$ such that $v+w\in\sigma_{j_1}$ and $w_1\not\in\sigma_{j_1}$. We claim that $v+w+w_1$ is contained in strictly fewer facets than $v+w$. Indeed, recall that for an arbitrary cone $\tau$, a face $\widetilde{\tau}$ of $\tau$, and $a,b\in\tau$, we have $a+b\in\widetilde{\tau}$ if and only if $a\in\widetilde{\tau}$ and $b\in\widetilde{\tau}$. Hence $v+w+w_1$ is not contained in \textit{more} facets than $v+w$, since any facet not containing $v+w$ cannot contain $v+w+w_1$. Similarly, since $w_1\not\in\sigma_{j_1}$ we have $v+w+w_1\not\in\sigma_{j_1}$. Hence $v+w+w_1$ is contained in strictly fewer facets than $v+w$, as desired. 
	
	By the induction hypothesis there is a face $\widetilde{\cH}_1$ of $Q$ and $\widetilde{w}_1\in\cF$ such that $\cF\subset \widetilde{\cH}_1$ and $v+w+w_1+\widetilde{w}_1\in\sigma_{\widetilde{\cH}_1}^{\mathrm{int}}$. Taking $\widetilde{w}=w_1+\widetilde{w}_1$ and $\widetilde{\cH}=\widetilde{\cH}_1$ we get the desired result.
\end{proof}

\noindent We now prove Theorem \ref{thing01}.

\begin{proof}[Proof of Theorem \ref{thing01}]\label{thing19}
	Fix $\ell\in\bbN$. Let $\nq{\ell}$ be as in Definition \ref{thing56} and let
	\begin{align*}
		e_0:=\max_{\dim\cH\ge\ell}\lceil\log_p m_{\cH}\rceil.
	\end{align*}
	 Fix $p>\nq{\ell}$. Let $\cF$ be an $(\ell-1)$-face of $Q$, and let $v\in M,e\in\bbN$ be such that $p^ev\in Q-\cF$. We wish to show that $p^{\widetilde{e}}v\in Q-\cF$ for $\widetilde{e}\ge e_0$.
	 
	To this end, let $w\in\cF$ be such that $p^ev+w\in Q$. By Lemma \ref{thing56}, we may assume there is a face $\cH$ of $Q$ such that $\cF\subset\cH$ and $v+w\in\sigma_{\cH}^{\mathrm{int}}$. We claim that $v\in\gr{\cH}$, and we will show this using Lemma \ref{thing55}.
	
	To apply Lemma \ref{thing55}, we must show that $p^ev\in\gr{\sigma_\cH\cap Q}$ and $v\in\gr{\sigma_\cH\cap M}$. We certainly have $v\in\gr{\sigma_\cH\cap M}$, since $v+w\in\sigma_\cH\cap M$ and $w\in\cF\subset\cH\subset\sigma_\cH\cap M$. So let's show $p^ev\in\gr{\sigma_\cH\cap Q}$.
	
	First, since $v\in\gr{\sigma_\cH\cap M}$ we also have $p^ev\in\gr{\sigma_\cH\cap M}$. But $p^ev+w\in Q$, so $p^ev+w\in\gr{\sigma_\cH\cap M}\cap Q=(\sigma_\cH\cap M)\cap Q=\sigma_{\cH}\cap Q$. Since $w\in\sigma_\cH\cap Q$, we get $p^ev\in\gr{\sigma_\cH\cap Q}$ as desired.
	
	We can now apply Lemma \ref{thing55} to get that $v\in\gr{\sigma_\cH\cap Q}=\gr{\cH}$.
	
	To prove the theorem, we now have two cases to consider. First suppose $\cF=\cH$. Then $v\in\gr{\cF}\subset Q-\cF$, which immediately gives the desired result. On the other hand, suppose $\cF\subsetneq\cH$. Then $v+w\in\sigma_\cH^{\mathrm{int}}\cap\gr{\cH}$, so $m(v+w)\in\cH$ for $m\ge m_\cH$ by Lemma \ref{thing1}. Since $mw\in\cF$ we get $mv\in Q-\cF$ for $m\ge m_\cH$. Now let $\widetilde{e}\ge e_0$. Note that $\dim\cH\ge\ell$ because $\cF\subsetneq\cH$. Hence $\widetilde{e}\ge\lceil\log_pm_\cH\rceil$, and so $p^{\widetilde{e}}\ge m_\cH$. Thus $p^{\widetilde{e}}v\in Q-\cF$ as desired.
\end{proof}

\noindent Finally, we can prove Theorem \ref{thing00}.

\begin{proof}[Proof of Theorem \ref{thing00}]\label{thing18}
	Fix $\ell\in\bbN$. If $\ell=0$ or $\ell>n$ then $H^\ell_\fm(k[Q])=0$ so the claim holds trivially. Thus we may assume $1\le\ell\le n$. Let
	\begin{align*}
		e_0:=\max_{\dim\cH\ge\ell}\lceil\log_p m_{\cH}\rceil.
	\end{align*}
	We wish to show that $\hsl{}{H^\ell_\fm(k[Q])}\le e_0$. 
	
	First recall that $H^\ell_\fm(k[Q])=\ker(\delta^{\ell})/\im(\delta^{\ell-1})$ where $\delta$ is the differential in the Ishida complex:
	\begin{align}
		0 \to k[Q]\to\bigoplus_{\text{rays }\cF}k[Q]_\cF\to\cdots\to\bigoplus_{i\text{-faces }\cF}k[Q]_\cF\xrightarrow{\delta^i}\cdots\to\bigoplus_{\text{facets }\cF}k[Q]_\cF\to k[M]\to0. \label{thing7}
	\end{align}
	Hence every cohomology class $[\alpha]\in H^\ell_\fm(k[Q])$ is represented by some $\alpha\in\bigoplus_{\ell\text{-faces }\cF}k[Q]_\cF$. The natural Frobenius action on the complex \eqref{thing7}, given by 
	\begin{align*}
		\bigoplus_{\ell\text{-faces }\cF}k[Q]_\cF&\to\bigoplus_{\ell\text{-faces }\cF}k[Q]_\cF\\
		\bigoplus_{\ell\text{-faces }\cF}\frac{f_\cF}{x^{w_\cF}}&\mapsto \bigoplus_{\ell\text{-faces }\cF}\frac{f_\cF^p}{x^{pw_\cF}},
	\end{align*}
	induces the Frobenius action on local cohomology. For an element $\alpha\in\bigoplus_{\ell\text{-faces }\cF}k[Q]_\cF$ we denote the images of $\alpha$ and $[\alpha]$ under these Frobenius actions by $\alpha^p$ and $[\alpha]^p$, respectively.
	
	The natural $Q$ (and therefore $M$) grading on $k[Q]$ induces an $M$-grading on all of the objects in the Ishida complex; furthermore the maps in the Ishida complex respect this grading.
	
	We must show that for a local cohomology class $[\alpha]\in H^\ell_\fm(k[Q])$, if $[\alpha]\in 0^F_{H^\ell_\fm(k[Q])}$ then $[\alpha]^{p^{\widetilde{e}}}=0$ for $\widetilde{e}\ge e_0$. Considering the $M$-grading on the Ishida complex, we may assume that $\alpha$ is homogeneous, i.e., $$\alpha=\bigoplus_{\ell\text{-faces }\cF}a_\cF x^v$$ for some $v\in M$ and $a_\cF\in k$.
	
	Since $[\alpha]^{p^e}=0$ we have $\alpha^{p^e}\in\im\delta^{\ell-1}$, i.e., $\alpha^{p^e}=\delta^{\ell-1}(\beta)$ for some $\beta\in\bigoplus_{\substack{(\ell-1)\\\text{faces }\cG}}k[Q]_\cG$. Write $$\beta=\bigoplus_{\substack{(\ell-1)\\\text{faces }\cG}}b_{\cG}x^{p^ev}$$ for some $b_{\cG}\in k$.
	
	Fix $\widetilde{e}\ge e_0$. If $\widetilde{e}\ge e$ then 
	\begin{align*}
		\alpha^{p^{\widetilde{e}}}=\left(\alpha^{p^e}\right)^{p^{\widetilde{e}-e}} = \left(\delta^{\ell-1}(\beta)\right)^{p^{\widetilde{e}-e}} = \delta^{\ell-1}(\beta^{p^{\widetilde{e}-e}}).
	\end{align*} 
	Hence $\alpha^{p^{\widetilde{e}}}\in\im\delta^{\ell-1}$, i.e., $[\alpha]^{p^{\widetilde{e}}}=0$.
	
	On the other hand, suppose $\widetilde{e}<e$. Let $\widetilde{\beta}=\bigoplus_{\substack{(\ell-1)\\\text{faces }\cG}}b_\cG^{p^{\widetilde{e}-e}}x^{p^{\widetilde{e}}v}$. We claim that $\widetilde{\beta}\in\bigoplus_{\substack{(\ell-1)\\\text{faces }\cG}}k[Q]_\cG$ and that $\alpha^{p^{\widetilde{e}}} = \delta^{\ell-1}(\widetilde{\beta})$.
	
	First we show that $\widetilde{\beta}\in\bigoplus_{\substack{(\ell-1)\\\text{faces }\cG}}k[Q]_\cG$. Since $k$ contains all $p^{e-\widetilde{e}}$ roots, it suffices to show that $p^{\widetilde{e}}v\in Q-\cG$ for all $(\ell-1)$-faces $\cG$ with $b_\cG\not=0$. Indeed, since $\beta\in\bigoplus_{\substack{(\ell-1)\\\text{faces }\cG}}k[Q]_\cG$ we have $p^ev\in Q-\cG$ for all $(\ell-1)$-faces $\cG$ such that $b_\cG\not=0$. Since $p>\nq{\ell}$, by Theorem \ref{thing01} we then have $p^{\widetilde{e}}v\in Q-\cG$ for all such $(\ell-1)$-faces $\cG$, as desired. Hence $\widetilde{\beta}\in\bigoplus_{\substack{(\ell-1)\\\text{faces }\cG}}k[Q]_\cG$.
	
	Now let's show that $\alpha^{p^{\widetilde{e}}} = \delta^{\ell-1}(\widetilde{\beta})$. Indeed,
	\begin{align*}
		\left(\delta^{\ell-1}(\widetilde{\beta})\right)^{p^e} &= \left(\delta^{\ell-1}\left(\bigoplus_{\substack{(\ell-1)\\\text{faces }\cG}}b_\cG^{p^{\widetilde{e}-e}}x^{p^{\widetilde{e}}v}\right)\right)^{p^e}\\
		&= \delta^{\ell-1}\left(\bigoplus_{\substack{(\ell-1)\\\text{faces }\cG}}b_\cG^{p^{\widetilde{e}}}x^{p^{\widetilde{e}+e}v}\right)\\
		&= \left(\delta^{\ell-1}\left(\bigoplus_{\substack{(\ell-1)\\\text{faces }\cG}}b_\cG x^{p^ev}\right)\right)^{p^{\widetilde{e}}}\\
		&= \left(\alpha^{p^e}\right)^{p^{\widetilde{e}}}\\
		&= \left(\alpha^{p^{\widetilde{e}}}\right)^{p^e}.
	\end{align*}
	Since Frobenius acts injectively on $\bigoplus_{\ell\text{-faces }\cF}k[Q]_\cF$ we have $\alpha^{p^{\widetilde{e}}}=\delta^{\ell-1}(\widetilde{\beta})$ as desired. Hence $[\alpha]^{p^{\widetilde{e}}}=0$ and the proof of the theorem is complete.
\end{proof}

\section{Bounds for Frobenius Test Exponents}\label{fte}

One area of interest is the connection between HSL numbers of local cohomology modules and Frobenius test exponents. In this section we see how the results of this paper can be used to find bounds on Frobenius test exponents in some cases.

For a Noetherian ring $R$ of characteristic $p>0$ and an ideal $I\subset R$, the \textit{Frobenius closure} of $I$ is 
\begin{align*}
	I^F &= \{r\in R\,|\,r^{p^e}\in I^{[p^e]}\text{ for some }e\in\bbN\},
\end{align*}
where $I^{[p^e]}$ denotes the ideal generated by all $\left(p^e\right)^{th}$ powers of elements in $I$. The \textit{Frobenius test exponent} of $I$, denoted $\fte{I}$, is the minimal $e\in\bbN$ such that $\left(I^F\right)^{[p^e]}=I^{[p^e]}$. Since $R$ is Noetherian, each ideal $I\subset R$ has a finite Frobenius test exponent. However, it may not be the case that there is a single $e\in\bbN$ that bounds $\fte{I}$ for \textit{all} ideals $I\subset R$. The question of the existence of such an $e$ was posed by Katzman-Sharp in \cite{KS06}, and a counterexample was found by Brenner \cite{brenner} for a two-dimensional domain which is standard graded over a field. However, the behavior is slightly nicer if we restrict to the set of parameter ideals of $R$. We thus define the \textit{Frobenius test exponent} of $R$ as
\begin{align}
	\fte{R} &= \sup\{\fte{\fq}\,|\,\fq\subset R\text{ a parameter ideal}\}\in\bbN\cup\{\infty\}. \label{thing10}
\end{align}
It has been shown that $\fte{R}<\infty$ for Cohen-Macaulay \cite{KS06}, generalized Cohen-Macaulay \cite{hksy06}, weakly $F$-nilpotent \cite{Quy19}, and generalized weakly $F$-nilpotent \cite{mad19} rings.

In several cases, bounds on the Frobenius test exponent of a local ring $(R,\fm)$ are given in terms of HSL numbers of the local cohomology of $R$ at $\fm$ (see \cite[Thm 2.4]{KS06},\cite[Main Thm]{Quy19},\cite[Thm 3.6]{mad19}). Theorem \ref{thing0} together with these known bounds may be used to give an explicit upper bound for $\fte{R}$ in the case that $R$ is any weakly $F$-nilpotent or Cohen-Macaulay affine semigroup ring. 

In particular, Katzman-Sharp \cite[Thm 2.4]{KS06} showed that if a local ring $(S,\fn)$ is Cohen-Macaulay then $\fte{S}=\hsl{}{H^{\dim S}_\fn(S)}$. Hence Theorem \ref{thing0} has the following corollary:
\begin{corollary}\label{thing8}
	With setup as in Theorem \ref{thing0}, if $k[Q]$ is Cohen-Macaulay then 
	\begin{align*}
		\fte{k[Q]}\le\max_{\cH}\lceil\log_p m_{\cH}\rceil,
	\end{align*}
	where the maximum ranges over all faces $\cH$ of $Q$.
\end{corollary} 

Furthermore, Quy \cite[Main Thm]{Quy19} showed that if a local ring $(S,\fn)$ of dimension $n$ is weakly $F$-nilpotent then 
\begin{align*}
	\fte{S}\le\sum_{j=0}^n\binom{n}{j}\hsl{}{H^j_\fn(S)}.
\end{align*}
Hence Theorem \ref{thing0} has the following corollary as well:

\begin{corollary}\label{thing11}
	With setup as in Theorem \ref{thing0}, if $k[Q]$ is weakly $F$-nilpotent then 
	\begin{align*}
		\fte{k[Q]}\le \sum_{j=0}^n\binom{n}{j}\max_{\dim\cH\ge j}\lceil\log_p m_{\cH}\rceil,
	\end{align*}
	where in the $j^{th}$ term in the sum the maximum is over all faces $\cH$ of $Q$ with $\dim\cH\ge j$.
\end{corollary}

\begin{example}
	Consider the semigroup $Q$ from Example \ref{thing14}. The semigroup ring $k[Q]$ is Cohen-Macaulay, so by Corollary \ref{thing8} and Example \ref{thing15} for $p>3$ the Frobenius test exponent (see Section \ref{intro}) satisfies
	\begin{align*}
		\fte{k[Q]}\le \max_\cH\ceil{\log_pm_\cH} = \max\{-\infty,\ceil{\log_p2}\} = \ceil{\log_p2}=1.
	\end{align*}
	Hence for any parameter ideal $\fq\subset k[Q]$, $\fq^F = \{r\in k[Q]\,|\, r^p\in \fq^{[p]}\}$ (note this includes the possibility that $\fq^F=\fq$). In the case that $\fte{k[Q]}=0$ we say that $k[Q]$ is \textit{parameter $F$-closed} \cite[Def 6.8]{qs17}. Hence the Hartshorne-Speiser-Lyubeznik number of $k[Q]$	is also zero \cite[Thm 3.4]{hq19}, i.e., $k[Q]$ is $F$-injective.
\end{example}

\begin{center}
	\textsc{acknowledgements}
\end{center}

The author would like to thank Kyle Maddox and Lance Miller for helpful discussion about this paper, and Karen Smith for her guidance. Thanks also to the University of Minnesota for allowing me to participate in the 2023 Macaulay2 Workshop \& Mini-School, and to Moty Katzman for his guidance during this workshop, which inspired this research. Also thanks to Olivia Strahan, Ben Baily, Anna Brosowsky, and Suchitra Pande for many helpful discussions. This work was conducted as part of my PhD thesis at the University of Michigan.

\bibliographystyle{alpha} 
\bibliography{refs} 

\end{document}